\documentclass{amsart}
\usepackage[latin1]{inputenc}
\usepackage{amsmath,amsfonts, amssymb, graphicx}

\newcommand{\surfint}{{\int_{\partial\Omega}}}
\newcommand{\BB}{{\mathbb{B}^d}}
\newcommand{\RR}{{\mathbb R}}
\newcommand{\TT}{{\mathrm T}}
\newcommand{\DD}{{\mathcal{T}}}
\renewcommand{\epsilon}{{\varepsilon}}

\newcommand{\nhat}{{\hat{n}}}
\newcommand{\that}{{\hat{t}}}

\newcommand{\rp}{{\rho^\prime}}
\newcommand{\rpp}{{\rho^{\prime\prime}}}
\newcommand{\Rp}{{R^{\prime}}}

\newcommand{\ainf}{{a_\infty}}
\newcommand{\ainfsq}{{a_\infty^2}}

\newcommand{\ostar}{{\omega^*}}
\newcommand{\Ostar}{{\Omega^*}}

\newcommand{\ubar}{{\overline{u}}}

\DeclareMathOperator{\grad}{grad}
\DeclareMathOperator{\mydiv}{div}

\newcommand{\sproj}{{P_{\partial\Omega}}}
\newcommand{\sgrad}{{\grad_{\partial\Omega}}}
\newcommand{\sdiv}{{\mydiv_{\partial\Omega}}}
\newtheorem{thm}{Theorem}
\newtheorem{lemma}[thm]{Lemma}
\newtheorem{prop}[thm]{Proposition}
\newtheorem{fact}[thm]{Fact}
\newtheorem{cor}[thm]{Corollary}

\theoremstyle{remark}
\newtheorem*{rmk}{Remark}

\newtheoremstyle{newdiffnotenonum}{}{}{\itshape}{}{\bfseries}{.}{ }%
  {\thmname{#1}\thmnote{( \mdseries #3)}}
\theoremstyle{newdiffnotenonum}

\title[Free plate isoperimetric inequality]{An isoperimetric inequality for fundamental tones of free plates}
\author{L. M. Chasman}
\address{Department of Mathematics, Knox College, Galesburg,
IL 61401, U.S.A.} \email{lchasman$@$knox.edu}
\date{\today}


\begin{document}
\begin{abstract}
We establish an isoperimetric inequality for the fundamental tone (first nonzero eigenvalue) of the free plate of a given area, proving the ball is maximal. Given $\tau>0$, the free plate eigenvalues $\omega$ and eigenfunctions $u$ are determined by the equation $\Delta\Delta u-\tau\Delta u = \omega u$ together with certain natural boundary conditions. The boundary conditions are complicated but arise naturally from the plate Rayleigh quotient, which contains a Hessian squared term $|D^2u|^2$.

We adapt Weinberger's method from the corresponding free membrane problem, taking the fundamental modes of the unit ball as trial functions. These solutions are a linear combination of Bessel and modified Bessel functions.
\end{abstract}

\maketitle

\section{\bf Introduction}
Laplacian and bi-Laplace operators are used to model many physical situations, with their eigenvalues representing quantities such as energy and frequency. The eigenvalues $\mu$ of the Neumann Laplacian on a region $\Omega$ determine the frequencies of vibration of a free membrane with that shape. If $\Omega^*$ is the ball of same volume as $\Omega$, then we have
\[
\mu_1(\Omega) \leq \mu_1(\Ostar)\qquad\text{with equality if and only if $\Omega$ is a ball.}
\]
First conjectured by Kornhauser and Stakgold \cite{KS52}, this isoperimetric inequality was proved for simply connected domains in $\RR^2$ by Szeg\H o \cite{S50,Serrata} and extended to all domains and dimensions by Weinberger \cite{W56}.

The main goal of this paper is to establish the analogous result for the eigenvalues $\omega$ the free plate under tension. That is, of all regions $\Omega$ with the same volume, we have the bound
\[
\omega_1(\Omega) \leq \omega_1(\Ostar)\qquad\text{with equality if and only if $\Omega$ is a ball.}
\]
Our proof relies on the variational characterization of eigenvalues with suitable trial functions. Similar to Weinberger's approach for the free membrane, we take our trial functions to be extensions of the fundamental mode of the unit ball. However, because the plate equation is fourth order, finding the trial functions and establishing the appropriate monotonicities is significantly more complicated than in the membrane case. The eigenmodes of the unit ball are identified in our companion paper \cite{chasman}, where we also establish several properties of these functions that are used in the proof of the isoperimetric inequality. If the reader is satisfied with a numerical demonstration, the needed properties of ultraspherical Bessel functions can be verified in any given dimension using Mathematica or Maple.

The boundary conditions of the free plate are not imposed, but instead arise naturally from the Rayleigh quotient. It is therefore extremely important that we begin with the correct Rayleigh quotient for the plate, which includes a Hessian term. These natural boundary conditions have long been known in the case $d=2$ (see, eg, \cite{RW74}); we include in this paper their derivation for all dimensions. 

The fundamental tone of the ball is extremal for other physically meaningful plate boundary conditions. The ball provides a lower bound for the clamped plate eigenvalues \cite{T81, N92, N95, AB95}. The methods used by Talenti, Nadirashvilli, Ashbaugh and Benguria to prove the clamped plate isoperimetric inequality are quite different than those for the free plate and membrane and only establish the bound in dimensions 2 and 3. The problem remains open for dimensions four and higher, with a partial result by Ashbaugh and Laugesen \cite{AL96}. For an overview of work on the clamped plate problem, see \cite[Chapter 11, p.\ 169--174]{H06} and \cite[p.\ 105--116]{Kesavan}.

Other plate boundary conditions include the simply supported plate, hinged plate, and Neumann boundary conditions. Plate problems are fourth-order and generally more difficult than their second-order membrane counterparts, because the theory of the bi-Laplace operator is not as well understood as the theory of the Laplacian. Verchota recently established the solvability of the biharmonic Neumann problem \cite{verchota}; these boundary conditions arise from the zero-tension plate and allow consideration of Poisson's ratio, a measure of a material property that we take to be zero for our free plate. Supported plate work includes Payne \cite{P58} and Licar and Warner \cite{LW73}, who examine domain dependence of plate eigenvalues. It would be natural to conjecture an isoperimetric inequality for the simply supported plate, although there does not seem to be any work on this problem to the best of our knowledge. Work with hinged plates includes Nazarov and Sweers \cite{NS07}. 

Other notable mathematical work on plates includes Kawohl, Levine, and Velte \cite{KLV}, who investigated the sums of low eigenvalues for the clamped plate under tension and compression, and Payne \cite{P58}, who considered both vibrating and buckling free and clamped plates and established inequalities bounding plate eigenvalues by their (free or fixed) membrane counterparts. For a broad survey of results, see \cite{Asummary,Bandle}.

\section{\bf Formulating the problem}
We now develop the mathematical formulation of the free plate isoperimetric problem.
Let $\Omega$ be a smoothly bounded region in $\RR^d$, $d\geq 2$, and fix a parameter $\tau > 0$. The ``plate'' Rayleigh quotient is
\begin{equation}
Q[u] = \frac{\int_\Omega |D^2 u|^2 + \tau |D u|^2\,dx}{\int_\Omega |u|^2\,dx}. \label{RQ}
\end{equation}
Here $|D^2u|=(\sum_{jk}u_{x_jx_k}^2)^{1/2}$ is the Hilbert-Schmidt norm of the Hessian matrix $D^2u$ of $u$, and $Du$ denotes the gradient vector.

Physically, when $d=2$ the region $\Omega$ is the shape of a homogeneous, isotropic plate. The parameter $\tau$ represents the ratio of lateral tension to flexural rigidity of the plate; for brevity we refer to $\tau$ as the tension parameter. Positive $\tau$ corresponds to a plate under tension, while taking $\tau$ negative would give us a plate under compression. The function $u$ describes a transverse vibrational mode of the plate, and the Rayleigh quotient $Q[u]$ gives the bending energy of the plate.

From the Rayleigh quotient \eqref{RQ}, we will derive the partial differential equation and boundary conditions governing the vibrational modes of a free plate. The critical points of \eqref{RQ} are the eigenstates for the plate satisfying the free boundary conditions and the critical values are the corresponding eigenvalues. The equation is:
\begin{equation}
\Delta \Delta u - \tau \Delta u = \omega u, \label{maineq}
\end{equation}
where $\omega$ is the eigenvalue, with the natural (\emph{i.e.}, unconstrained or ``free'') boundary conditions on $\partial\Omega$:
\begin{align}
&Mu := \frac{\partial^2 u}{\partial n^2}= 0 \label{BC1}\\
&Vu := \tau\frac{\partial u}{\partial n}-\sdiv\left(\sproj\left[(D^2u)n\right]\right)-\frac{\partial(\Delta u)}{\partial n} = 0\label{BC2}
\end{align}
Here $n$ is the outward unit normal to the boundary and $\sdiv$ and $\sgrad$ are the surface divergence and gradient. The operator $\sproj$ projects onto the space tangent to $\partial\Omega$.

We will prove in a later section that the spectrum of the Rayleigh quotient $Q$ is discrete, consisting entirely of eigenvalues with finite multiplicity:
\[
 0=\omega_0<\omega_1\leq\omega_2\leq\dots\rightarrow\infty.
\]
 We also have a complete $L^2$-orthonormal set of eigenfunctions $u_0\equiv$ const, $u_1$, $u_2$, and so forth.

We call $u_1$ the \emph{fundamental mode} and the eigenvalue $\omega_1$ the \emph{fundamental tone}; the latter can be expressed using the Rayleigh-Ritz variational formula:
\[
\omega_1(\Omega)=\min\{Q[u]:u\in H^2(\Omega), \int_\Omega u\,dx=0\}.
\]
In general, the $k$th eigenvalue is the minimum of $Q[u]$ over the space of all functions $u$ $L^2$-orthogonal to the eigenfunctions $u_0$, $u_1$,$\dots$, $u_{k-1}$. Because $u_0$ is the constant function, the condition $u\perp u_0$ can be written $\int_\Omega u\,dx=0$. Note that in the limiting case $\tau=0$, the first $d+1$ eigenvalues of $\Omega$ are trivial because $Q[u]=0$ for all linear functions $u$. We therefore need the tension parameter $\tau$ to be positive in order to have a nontrivial inequality.

The eigenvalue equation \eqref{maineq} can also be obtained by separating the plate wave equation
\[
\phi_{tt}=-\Delta\Delta\phi+\tau\Delta\phi,
\]
by the separation $\phi(x,t)=u(x)\cos(\sqrt{\omega}t)$. The eigenvalue $\omega$ is therefore the square of the frequency of vibration of the plate. 

\section{\bf Main result}
The main goal of this paper is to prove an isoperimetric inequality for the fundamental tone of the free plate under tension. Let $\Omega^*$ denote the ball with the same volume as our region $\Omega$.

\begin{thm}\label{thm1} For all smoothly bounded regions of a fixed volume, the fundamental tone of the free plate with a given positive tension is maximal for a ball. That is, if $\tau>0$ then the first nonzero eigenvalue $\omega$ of $\Delta\Delta u-\tau\Delta u=\omega u$ subject to the natural boundary conditions \eqref{BC1} and \eqref{BC2} satisfies
\begin{equation}
\omega_1(\Omega) \leq \omega_1(\Omega^*), \qquad\text{with equality if and only if $\Omega$ is a ball.} \label{isoineq}
\end{equation}
\end{thm}

\begin{figure}
\begin{center}
\includegraphics{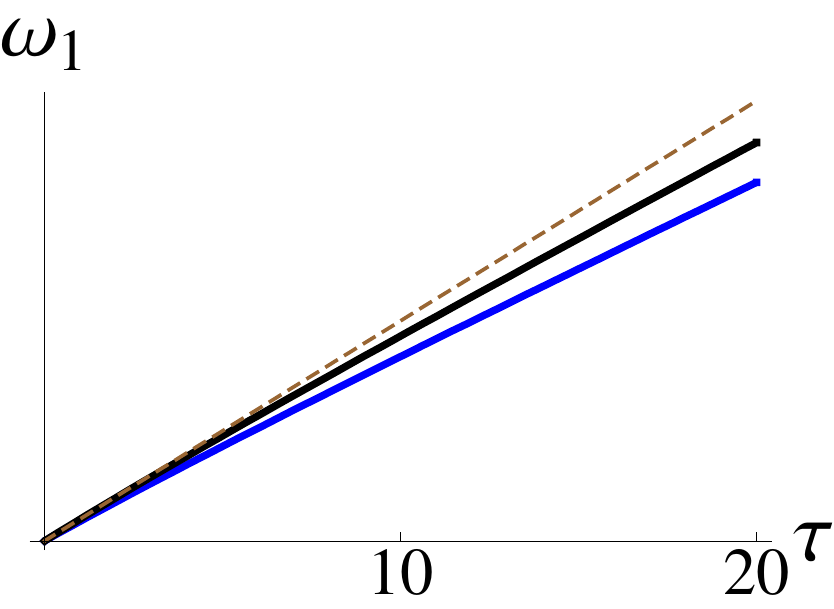}
\caption{The fundamental tone $\omega_1(\Ostar)$ of the disk (middle curve) and of another region $\omega_1(\Omega)$ (bottom curve). The dashed straight line is tangent to $\omega_1(\Ostar)$ at $\tau=0$.}
\end{center}
\end{figure}

The proof of Theorem~\ref{thm1} will proceed from a series of lemmas, following roughly this outline. A more detailed summary follows.
\begin{itemize}
\item Section 8 -- Defining the trial functions, showing concavity of the radial part of the trial function, and evaluating the Rayleigh quotient
\item Section 9 -- Proving partial monotonicity of the Rayleigh quotient
\item Section 10 -- Establishing rescaling and rearrangement results, and proving the theorem.
\end{itemize}

Adapting Weinberger's approach for the membrane \cite{W56}, we construct in Lemma~\ref{trialfcn} trial functions with radial part $\rho$ matching the radial part of the fundamental mode of the ball. We follow by proving in Lemma~\ref{gppneg} a concavity property of $\rho$ that will be needed later on. We next bound the eigenvalue $\omega$ by a quotient of integrals over our region $\Omega$, both of whose integrands are radial functions (Lemma~\ref{lemmaboundRC}). These integrands will be shown to have a "partial monotonicity". The denominator's integrand is increasing by Lemma~\ref{mondenom} and the numerator's integrand satisfies a decreasing partial monotonicity condition by Lemma~\ref{monnum}.

The proof of Lemma~\ref{monnum} becomes rather involved and so is contained in its own section and broken into two cases, Lemma~\ref{largetau} for large $\tau$ values, and Lemma~\ref{smalltau} for small values of $\tau$. The latter in turn requires some facts about particular polynomials, proved in Lemmas~\ref{poly1} and~\ref{poly2}. We then exploit partial monotonicity to see that the quotient of integrals is bounded above by the quotient of the same integrals taken over $\Ostar$, by Lemma~\ref{monint}. Finally, we conclude that the quotient of integrals on $\Ostar$ is in fact equal to the eigenvalue $\ostar$ of the unit ball. From there we deduce the theorem.

\section{\bf Existence of the spectrum and regularity of solutions}\label{spectrum}
Our first task is to investigate the spectrum of the fourth-order operator associated with our Rayleigh quotient $Q$ in \eqref{RQ}. In this section we show the spectrum is entirely discrete, with an associated weak eigenbasis. We will then establish regularity of the eigenfunctions up to the boundary and derive the natural boundary conditions.

For this section only we will allow $\tau$ to be any real number. We continue to require $\Omega\subset\RR^d$ to be smoothly bounded unless otherwise stated.

\subsection*{The existence of the spectrum}
We consider the sesquilinear form
\begin{align*}
a(u,v) &=\int_\Omega\sum_{i,j=1}^d\overline{u_{x_ix_j}}v_{x_ix_j}+\tau(\overline{D u}\cdot D v)\,dx
\end{align*}
in $L^2(\Omega)$ with form domain $H^2(\Omega)$. Note the plate Rayleigh quotient $Q$ can be written in terms of $a$, with $Q[u]=a(u,u)/\|u\|_{L^2}^2$.

\begin{prop} \label{spect}The spectrum of the operator $\Delta^2-\tau\Delta$ associated with the form $a(\cdot,\cdot)$ consists entirely of isolated eigenvalues of finite multiplicity $\omega_0\leq\omega_1\leq\omega_2\leq\dots\to\infty$. There exists an associated set of real-valued weak eigenfunctions which is an orthonormal basis for $L^2(\Omega)$.
\end{prop}

\begin{proof}
By Cauchy-Schwarz, the form $a(\cdot,\cdot)$ is bounded on $H^2(\Omega)$, and so is continuous. We will show the quadratic form $a(u,u)$ is coercive; that is, for some positive constants $c_1$ and $c_2$  $a(u,u)+c_1\|u\|^2\geq c_2\|u\|^2_{H^2(\Omega)}$. By the boundedness of $a$ on $H^2$, this is equivalent to showing the norm associated with $a$, 
\[
 \|u\|_a^2=a(u,u)+c_1\|u\|^2
\]
is equivalent to $\|\cdot\|^2_{H^2(\Omega)}$, and hence $a$ is closed on $H^2(\Omega)$. Because $\Omega$ is smoothly bounded, $H^(\Omega)$ and $L^2(\Omega)$ can be extended to $H^2(\RR^d)$ and $L^(\RR^d)$ respectively. The space $H^2(\RR^d)$ is compactly embedded in $L^(\RR^d)$. Then by a standard result (see e.g., Corollary 7.D \cite[p. 78]{SH77}), the form $a$ has a set of weak eigenfunctions which is an orthonormal basis for $L^2(\Omega)$, and the corresponding eigenvalues are of finite multiplicity and satisfy
\begin{equation}\label{eigenvalueineq}
 \omega_1\leq\omega_2\leq \dots\leq\omega_n\rightarrow\infty\quad\text{as}\quad n\rightarrow\infty.
\end{equation}

For $\tau>0$, coercivity of the form $a$ is easily proved:
\begin{align*}
a(u,u) + \tau \|u\|^2&\geq \|D^2u\|^2+\tau\|D u\|^2 +\tau \|u\|^2\\
&\ge \min(\tau,1)\|u\|_{H^2}^2,
\end{align*}
where all unlabeled norms are $L^2$ norms on $\Omega$. 

To prove coercivity when $\tau \leq 0$, we must somehow arrive at a positive constant in front of the $|Du|^2$ term. We cannot use Poincar\'e's inequality on the $|D^2u|$ term as this will introduce terms involving the average value of $Du$. Instead, we will exploit an interpolation inequality.

By Theorem 7.28 of \cite[p. 173]{GT}, we have that for any index $1\leq j\leq n$ and any $\epsilon>0$,
\begin{equation}
\|\partial_{x_j}u\|_{L^2(\Omega)}^2 \leq \epsilon\|u\|_{H^2(\Omega)}^2+C\epsilon^{-1}\|u\|_{L^2(\Omega)}^2
\end{equation}
with $C=C(\Omega)$ a constant. Replacing $\epsilon$ by $\epsilon/d$ and summing over $j$, we see
\begin{equation*}
\|D^2u\|_{L^2}^2 \geq \left(\frac{1}{\epsilon}-1\right)\|Du\|_{L^2}^2
-\left(\frac{C}{\epsilon^2}+1\right)\|u\|_{L^2}^2.
\end{equation*}
Fix $\delta\in(0,1)$. Let $K > 0$. Then
\begin{align*}
a(u,u)&+K\|u\|_{L^2}^2= \|D^2u\|_{L^2}^2-|\tau|\|Du\|^2_{L^2}+K\|u\|_{L^2}^2\\
&\geq (1-\delta)\|D^2u\|_{L^2}^2+\left(\frac{\delta}{\epsilon}-\delta-|\tau|\right)\|Du\|^2_{L^2}+\left(K-\frac{C\delta}{\epsilon^2}-\delta\right)\|u\|_{L^2}^2\\
&\geq \min\left\{1-\delta,\frac{\delta}{\epsilon}-\delta-|\tau|,K-\frac{C\delta}{\epsilon^2}-\delta\right\}\|u\|_{H^2},
\end{align*}
We can choose our $\epsilon$ small and our $K$ large so that the minimum is positive, which proves coercivity. For example, for $\delta=1/2$, we need to take $\epsilon<1/(1+2|\tau|)$ and $K>\frac{1}{2}\Big(C+1+2|\tau|\Big)$. Thus $a$ is coercive for all $\tau$.

Now suppose $u$ is a weak eigenfunction corresponding to eigenvalue $\omega$. Because $\omega$ is real-valued, by taking the complex conjugate of the weak eigenvalue equation we see that $\ubar$ is also a weak eigenfunction with the same eigenvalue. Thus the real and imaginary parts of $u$ are both eigenfunctions associated with $\omega$, and we may choose our eigenfunctions to be real-valued.
\end{proof}

Note that for any bounded region $\Omega$ and all real values of $\tau$, the constant function solves the weak eigenvalue equation with eigenvalue zero. For all nonnegative values of $\tau$, the Rayleigh quotient is nonnegative for all functions and so $0=w_0\leq w_1\leq\cdots\leq w_k\leq\dots$.  When $\tau=0$, the coordinate functions $x_1,\dots,x_d$ are also solutions with eigenvalue zero, and so the lowest eigenvalue is at least $d+1$-fold degenerate, as noted in the introduction. Taking instead $\tau>0$, the Raleigh quotient shows that the fundamental tone $\omega_1$ is positive, and so we have:
\[
 0=\omega_0<\omega_1\leq\omega_2\leq\cdots\leq\omega_n\rightarrow\infty\quad\text{as}\quad n\rightarrow\infty.
\]

\subsection*{Regularity}
We aim to establish regularity of the weak eigenfunctions by appealing to interior and boundary regularity theory for elliptic operators. 

\begin{prop} \label{regprop} For any $\tau \in\RR$ and smoothly bounded $\Omega$, the weak eigenfunctions of $\Delta\Delta-\tau\Delta$ are smooth on $\overline{\Omega}$. 
\end{prop}

\begin{proof}
Let $u$ be a weak eigenfunction of $A$ with associated eigenvalue $\omega$; by Proposition~\ref{spect} we have $u\in\DD(a)=H^2(\Omega)$. Then by a theorem in \cite[p 668]{Nir55}, we have $u\in H^k(\Omega)$ for every positive integer $k$. Thus we have $u\in H^k(\Omega)$ for all $k\in \mathbb{Z}^+$, and so $u\in C^\infty(\Omega)$.

Regularity on the boundary follows from global interior regularity and the Trace Theorem (see, for example, \cite[Prop 4.3, p. 286 and Prop 4.5, p. 287.]{taylor}). Thus we have $u\in C^\infty(\overline{\Omega})$, as desired.
\end{proof}

\section{\bf The Natural Boundary Conditions}
In this section, our goal is to derive the form of the natural boundary conditions necessarily satisfied by all eigenfunctions. Consider the weak eigenvalue equation for eigenfunction $u$ with eigenvalue $\omega$ and some test function $\phi\in C^\infty_c(\Omega)$:
\[
 \int_\Omega \sum_{i,j}u_{x_ix_j}\phi_{x_ix_j}+\tau Du\cdot D\phi -\omega u\phi\,dx.
\]
Because the eigenfunction $u$ is smooth, we may use integration by parts to move most of the derivatives on $\phi$ to $u$; this gives us a volume integral and two surface integrals that must vanish for all $\phi$.

The natural boundary conditions are rather complicated in higher dimensions, and so we state the two-dimensional case first. The boundary conditions in this case have been known for some time: see, for example, \cite{RW74}
\begin{prop} \label{2dimBC}(Two dimensions)
For $\Omega\subset\RR^2$, the natural boundary conditions for eigenfunctions of the free plate under tension have the form
\begin{align*}
&Mu := \frac{\partial^2 u}{\partial n^2} = 0 \\
&Vu :=\tau \frac{\partial u}{\partial n}-\frac{\partial (\Delta u)}{\partial n} - \frac{\partial }{\partial s} \left(\frac{\partial^2u}{\partial s\partial n}-K(s)\frac{\partial u}{\partial s}\right)= 0
\end{align*}
where $n$ denotes the outward unit normal derivative, $s$ the arclength, and $K$ the curvature of $\partial\Omega$.
\end{prop}
We also look at one example of the natural boundary conditions for a region with corners. Notice that an additional condition arises at the corners!
\begin{prop} (Rectangular region in two dimensions) \label{rectBC}
When $\Omega\subset\RR^2$ is a rectangular region with edges parallel to the coordinate axes, the natural boundary conditions for eigenfunctions of the free plate under tension have the form
\begin{align*}
&\frac{\partial^2 u}{\partial n^2} = 0 \quad\text{at each edge}\\
&\tau \frac{\partial u}{\partial n}-\frac{\partial^3 u}{\partial s^2\partial n}-\frac{\partial (\Delta u)}{\partial n}=0\qquad\text{on each edge}\\
&u_{xy}=0 \qquad\text{at each corner}
\end{align*}
where $n$ and $s$ indicate the normal and tangent directions.
\end{prop}
Finally, we state the natural boundary conditions for a smoothly-bounded region in higher dimensions:
\begin{prop} (General) \label{generalBC} For any smoothly bounded $\Omega$, the natural boundary conditions for eigenfunctions of the free plate under tension have the form
\begin{align*}
&Mu := \frac{\partial^2 u}{\partial n^2} = 0 &\text{on $\partial\Omega$,}\\
&Vu := \tau\frac{\partial u}{\partial n}
-\sdiv\Big(\sproj\left[(D^2u)n\right]\Big)-\frac{\partial\Delta u}{\partial n} = 0 &\text{on $\partial\Omega$,}
\end{align*}
where $n$ denotes the normal derivative and $\sdiv$ is the surface divergence. The projection $\sproj$ projects a vector $v$ at a point $x$ on $\partial\Omega$ into the tangent space of $\partial\Omega$ at $x$.
\end{prop}
\begin{proof}[Proof of Proposition~\ref{generalBC}]
Our eigenfunctions $u$ are smooth on $\overline{\Omega}$ by Proposition 2 and satisfy the weak eigenvalue equation $a(u,\phi)-\omega(u,\phi)_{L^2(\Omega)}=0$ for all $\phi\in H^2(\Omega)$. That is,
\[
 \int_\Omega \left(\sum_{i,j=1}^du_{x_ix_j}\phi_{x_ix_j}+\tau D\phi\cdot Du-\omega u\phi\right)\,dx=0.
\]

As in the membrane case, we make much use of integration by parts. Let $n$ denote the outward unit normal to the surface $\partial\Omega$. To simplify our calculations, we consider each term separately.

The gradient term only needs one use of integration by parts:
\[
\int_\Omega Du\cdot D\phi\,dx=\surfint\phi \frac{\partial u}{\partial n}\,dS -\int_\Omega\phi(\Delta u)\,dx.
\]

The Hessian term becomes:
\begin{align*}
\int_\Omega&\sum_{i,j}u_{x_ix_j}\phi_{x_ix_j}\,dx\\
&=\surfint\left(D\phi\cdot\Big((D^2u)n\Big)-\phi\frac{\partial(\Delta u)}{\partial n}\right)\,dS +\int_\Omega(\Delta^2 u)\phi\,dx,
\end{align*}
after integrating by parts twice.

We wish to transform the term involving $D\phi$ in the above surface integral using integration by parts. Because we are on $\partial\Omega$, we must treat the normal and tangential components separately. We can then use the Divergence theorem for integration on $\partial\Omega$.

We note that the surface gradient $\sgrad$ equals $D - n\partial_n$ when applied to a function (like $\phi$) that is defined on a neighborhood of the boundary. Thus $\sgrad \phi$ gives the tangential part of the Euclidean gradient vector. Hence,
\begin{align*}
\surfint &D\phi\cdot\Big((D^2u)n\Big)\,dS \\
&= \surfint \left(n\frac{\partial\phi}{\partial n}+\sgrad\phi\right)\cdot\left(n\frac{\partial^2u}{\partial n^2}+\sproj\left[(D^2u)n\right]\right)\,dS\\
&=\surfint \frac{\partial\phi}{\partial n}\frac{\partial^2u}{\partial n^2} +\Big\langle\sgrad\phi,\sproj\left[(D^2u)n\right]\Big\rangle_{\partial\Omega}\,dS\\
&= \surfint \frac{\partial\phi}{\partial n} \frac{\partial^2u}{\partial n^2}-\phi\,\sdiv\left(\sproj\left[(D^2u)n\right]\right)\,dS,
\end{align*}
by the Divergence Theorem on the surface $\partial\Omega$. Here $\langle\cdot,\cdot\rangle_{\partial\Omega}$ denotes the inner product on the tangent space to $\partial\Omega$. Recall $\sproj$ projects a vector at a point $x$ on $\partial\Omega$ onto the tangent space of $\partial\Omega$ at $x$.

Thus for $u$ an eigenfunction associated with eigenvalue $\omega$, we see
\begin{align*}
0&=\int_\Omega\phi\Big(\Delta^2 u-\tau\Delta u-\omega u\Big)\,dx\\
&\qquad+\surfint \frac{\partial\phi}{\partial n} \frac{\partial^2u}{\partial n^2}+\phi\left(\tau\frac{\partial u}{\partial n}-\frac{\partial\Delta u}{\partial n}-\sdiv\Big(\sproj\left[(D^2u)n\right]\Big) \right)\,dS.
\end{align*}

As in the membrane case, this identity must hold for all $\phi\in H^2(\Omega)$. If we take any compactly supported $\phi$, then the volume integral must vanish; because $\phi$ is arbitrary, we must therefore have $\Delta^2 u-\tau\Delta u-\omega u=0$ everywhere. Similarly, the terms multiplied by $\phi$ and $\partial\phi/\partial n$ must vanish on the boundary. Collecting these results, we obtain the eigenvalue equation \eqref{maineq} and natural boundary conditions of Proposition~\ref{generalBC}.
\end{proof}

\begin{proof}[Proof of Proposition~\ref{2dimBC}]
Here $d=2$; take rectangular coordinates $(x,y)$. We parametrize $\partial\Omega$ by arclength $s$ and define coordinates $(n,s)$, with $n$ the normal distance from $\partial\Omega$, taken to be positive outside $\Omega$. Write $\nhat(s)$ and $\that(s)$ for the outward unit normal and unit tangent vectors to the boundary. Then $\sproj\left[f_1\nhat+f_2\that\right]=f_2\that$ and the operators $\sdiv$ and $\sgrad$ both simply take the derivative with respect to arclength $s$. That is, for a scalar function $f(s)$, and taking $t(s)$ to be the tangent vector to the surface, we have
\[
 \sgrad f(s)=f'(s) \qquad\text{and}\qquad \sdiv(f(s)\that(s))=f'(s).
\]
 and so we may write
\[
 \sdiv\Big(\sproj\left[(D^2u)n\right]\Big)= \frac{\partial}{\partial s}t^\TT(D^2u)n.
\]
The tangent line to $\partial\Omega$ at the point $(0,s)$ in our new coordinates forms an angle $\alpha=\alpha(s)$ with the $x$-axis (see
\cite[p. 230] {RW74}); the curvature of $\partial\Omega$ is given by $K(s)=\alpha^\prime(s)$. Then in rectangular coordinates, the unit tangent vector is $(\cos\alpha,\sin\alpha)$, and the outward unit normal is $(\sin\alpha,-\cos\alpha)$. Thus we have
\[
 \frac{\partial}{\partial s}t^\TT(D^2u)s =\partial_s\Big(\sin\alpha\cos\alpha(u_{xx}-u_{yy})+(\sin^2\alpha-\cos^2\alpha)u_{xy}\Big).
\]
By \cite[p. 233]{RW74}, on $\partial\Omega$ under our change of coordinates, we have
\begin{align*}
 u_{xx}&=u_{nn}\sin^2\alpha+u_{ss}\cos^2\alpha+2u_{ns}\sin\alpha\cos\alpha+Ku_n\cos^2\alpha-2Ku_s\sin\alpha\cos\alpha\\
 u_{yy}&=u_{nn}\cos^2\alpha+u_{ss}\sin^2\alpha-2u_{ns}\sin\alpha\cos\alpha+Ku_n\sin^2\alpha+2Ku_s\sin\alpha\cos\alpha\\
 u_{xy}&=-u_{nn}\cos\alpha\sin\alpha+u_{ss}\cos\alpha\sin\alpha+u_{ns}(\sin^2\alpha-\cos^2\alpha)\\
 &\qquad+Ku_n\cos\alpha\sin\alpha-Ku_s(\sin^2\alpha-\cos^2\alpha).
\end{align*}
So after simplification,
\[
\frac{\partial}{\partial t}[n^T (D^2u)t]=\frac{\partial}{\partial s}\left(\frac{\partial^2u}{\partial s\partial n}-K(s)\frac{\partial u}{\partial s}\right).
\]
This together with the results of Proposition~\ref{generalBC} yields the form of $Vu$ given in Proposition~\ref{2dimBC}. $Mu$ is unchanged, and so this completes the proof.
\end{proof}

\begin{proof}[Proof of Proposition~\ref{rectBC}]
Our previous findings do not completely apply because $\partial\Omega$ has corners, although our argument proceeds similarly. For convenience of notation, we will take $\Omega$ to be the square $[0,1]^2$.

The Hessian term gives us a condition at the corners. In particular, after integrating by parts twice, we have:
\begin{align*}
\int_\Omega &u_{xx}\phi_{xx}+2u_{xy}\phi_{xy}+u_{yy}\phi_{yy}\,dA\\
&=\int_\Omega \phi\Big(u_{xxxx}+2u_{xxyy}+u_{yyyy}\Big)\,dA\\
&\quad+\int_{0}^{1}\Big(u_{xx}\phi_x-u_{xxx}\phi+u_{xy}\phi_y-u_{xyy}\phi\Big)\Big|_{x=0}^{x=1}\,dy\\
&\quad+\int_{0}^{1}\Big(u_{yy}\phi_y-u_{yyy}\phi+u_{xy}\phi_x-u_{xxy}\phi\Big)\Big|_{y=0}^{y=1}\,dx.
\end{align*}
Since
\begin{align*}
 \int_0^1u_{xy}\phi_y\,dy=u_{xy}\phi\Big|_{y=0}^{y=1}-\int_0^1u_{xyy}\phi\,dy
\end{align*}
and 
\begin{align*}
 \int_0^1u_{xy}\phi_x\,dx=u_{xy}\phi\Big|_{x=0}^{x=1}-\int_0^1u_{xxy}\phi\,dx
\end{align*}
we obtain
\begin{align*}
 \int_\Omega &u_{xx}\phi_{xx}+2u_{xy}\phi_{xy}+u_{yy}\phi_{yy}\,dA\\
&=\int_\Omega \phi\Big(u_{xxxx}+2u_{xxyy}+u_{yyyy}\Big)\,dA\\
&\quad+\int_{0}^{1}\Big(u_{xx}\phi_x-\phi(2u_{xyy}+u_{xxx})\Big)\Big|_{x=0}^{x=1}\,dy\\
&\quad+\int_{0}^{1}\Big(u_{yy}\phi_y-\phi(2u_{xxy}+u_{yyy})\Big)\Big|_{y=0}^{y=1}\,dx\\
&\qquad+2u_{xy}\Big|_{x=0}^{x=1}\Big|_{y=0}^{y=1}.
\end{align*}

Because the Divergence Theorem does apply to regions with piecewise-smooth boundaries, the gradient term is the same as in the smooth-boundary case. The final term above is the only term that depends only on the behavior of $u$ and $\phi$ at the corners; arguing as before, we obtain the eigenvalue equation and natural boundary conditions, with the additional condition
\[
0=u_{xy}\phi\Big|_{x=0}^{1}\Big|_{y=0}^{1}.
\]
That is, we must have $u_{xy}=0$ at the corners.\end{proof}

\section*{Example: natural boundary conditions on the ball}
When $\Omega$ is a ball, we can simplify the general boundary conditions.

\begin{prop}\label{ballBC} (Ball) The natural boundary conditions in the case $\Omega=\BB(R)$, the ball of radius $R$, are
\begin{align}
&Mu := u_{rr} = 0 &\text{at $r=R$,}\label{BCb1}\\
&Vu := \tau u_r-\frac{1}{r^2}\Delta_S\left(u_r-\frac{u}{r}\right)-(\Delta u)_r = 0&\text{at $r=R$.}\label{BCb2}
\end{align}
\end{prop}

\begin{proof}
When $\Omega$ is a ball, the normal vector to the surface at a point $x$ is $n = x/R$. Then the $i$th component of $(D^2u)n$ is given by
\[
 \sum_{j=1}^du_{x_ix_j}\frac{x_j}{R}
\]
and can be rewritten as
\[
 \frac{1}{R}\frac{\partial}{\partial x_i}\left(\sum_{j=1}^du_{x_j}x_j-u\right).
\]
Therefore,
\[
 (D^2u)n=D\left(Du\cdot\frac{x}{R}-\frac{u}{R}\right).
\]
Then the projection $\sproj$ takes the tangential component of the above gradient vector, and so
\[
 \sproj\left[(D^2u)n\right] =\grad_{\partial\BB(R)}\left(Du\cdot\frac{x}{R}-\frac{u}{R}\right).
\]
We know $\sdiv\sgrad=\Delta_{\partial\Omega}$ by definition. For the ball of radius $R$, we have $\Delta_{\partial\BB(R)}=\frac{1}{R^2}\Delta_S$. The operator $\Delta_S$ is the spherical Laplacian, consisting of the angular part of the Laplacian. It satisfies the identity $\Delta=\frac{\partial^2}{\partial r^2}+\frac{d-1}{r}\frac{\partial}{\partial r}+\frac{1}{r^2}\Delta_S$.

Thus
\begin{align*}
 \sdiv\sproj\left[(D^2u)n\right]&=\Delta_{\partial\BB(R)}\left(Du\cdot\frac{x}{R}-\frac{u}{R}\right)\\
&=\Delta_{\partial\BB(R)}\left(u_r-\frac{u}{R}\right),
\end{align*}
by noting that $Du\cdot n=\partial u/\partial r$.  The boundary conditions of Proposition~\ref{generalBC} then simplify to \eqref{BCb1} and \eqref{BCb2}, as desired.\end{proof}

\subsection*{The one-dimensional case}
The one-dimensional analog of the free plate is the free rod, represented by an interval $I=[a,b]$ on the real line. We include its boundary conditions for the sake of completeness.
 
We may derive the natural boundary conditions from the weak eigenvalue equation as before. We obtain as boundary conditions
\begin{align*}
 u''\Big|_a^b&=0\\
\tau u' - u'''\Big|_a^b&=0
\end{align*}
and the eigenvalue equation
\[
 u''''-\tau u''=\omega u.
\]
Note that these are in fact the one-dimensional analogues of the boundary conditions and eigenvalue equation obtained in Proposition~\ref{generalBC}. The computations are straightforward integrations by parts and thus omitted.

The fundamental tone of the free plate in dimensions $d\geq 2$ had simple angular dependence; the fundamental tone of the free rod under tension can be proved to be an odd function. See \cite[Chapter 7]{cthesis}. 

Note that we do not have an isoperimetric inequality for the free rod, because all connected domains of the same area are now intervals of the same length, and are identical up to translation.  

\section{\bf The fundamental tone as a function of tension}\label{fundmodtens}
Fix the smoothly bounded domain $\Omega$. We will estimate how the fundamental tone $\omega_1=\omega_1(\tau)$ depends on the tension parameter $\tau$, establishing bounds used in the proof of Theorem~\ref{thm1}. We will also examine the behavior of $\omega_1$ in the extreme case as $\tau\to\infty$. 

First we note that the Rayleigh quotient \eqref{RQ} is linear and increasing as a function of $\tau$. Our eigenvalue $\omega_1(\tau)$ is the infimum of $Q[u]$ over $u\in H^2(\Omega)$ with $\int_\Omega u\,dx=0$, and thus $\omega_1(\tau)$ is itself a concave, increasing function of $\tau$.

Next, we will prove $\omega_1(\tau)/\tau$ is bounded above and below for all $\tau>0$. Recall $\mu_1$ is the fundamental tone of the free membrane.

\begin{lemma}\label{wbounds} For all $\tau\geq 0$ we have
 \begin{equation}\label{lem31}
  \tau\mu_1 \leq\omega_1(\tau)\leq \tau \frac{|\Omega|d}{\int_\Omega |x-\bar{x}|^2 \, dx},
 \end{equation}
where $\bar{x}=\int_\Omega x\,dx/|\Omega|$ is the center of mass of $\Omega$. In particular, when $\Omega$ is the unit ball,
 \begin{equation} 
  \tau\mu_1 \leq \omega_1(\tau) \leq \tau(d+2) .\label{omegabounds}
 \end{equation}
 Furthermore, the upper bounds in \eqref{lem31} and \eqref{omegabounds} hold for all $\tau\in\RR$.
\end{lemma}
These bounds are illustrated in Figure~\ref{Lemma31bounds}.
\begin{proof}
To establish the upper bound, take the coordinate functions as trial functions: $u_k=x_k-\bar{x}_k$, for $k=1,\dots,d$. Note $\int_\Omega u_k\,dx=0$ by definition of center of mass, so the $u_k$ are valid trial functions. All second derivatives of the $u_k$ are zero, so we have
\[
 \omega_1(\tau) \leq Q[u_k] =\frac{\int_\Omega\tau|Du_k|^2\,dx}{\int_\Omega u_k^2\,dx}
	=\tau\frac{\int_\Omega 1\,dx}{\int_\Omega (x_k-\bar{x}_k)^2\,dx}.
\]
Clearing the denominator and summing over all indices $k$, we obtain
\[
 \omega_1(\tau)\int_\Omega |x-\bar{x}|^2\,dx \leq \tau |\Omega| d,
\]
which is the desired upper bound. When $\Omega$ is the unit ball, note $\int_\Omega |x|^2\,dx=|\Omega|d/(d+2)$.

Now we treat the lower bound. Let $u\in H^2(\Omega)$ with $\int_\Omega u\,dx=0$. Then
\[
 Q[u]\geq\frac{\tau\int_\Omega |Du|^2\, dx}{\int_\Omega u^2\,dx} \geq \tau \mu_1
\]
by the variational characterization of $\mu_1$. Taking the infimum over all trial functions $u$ for the plate yields $ \omega_1(\tau)\geq \tau\mu_1$.
\end{proof}
\begin{figure}[t!]
\begin{center}
\includegraphics{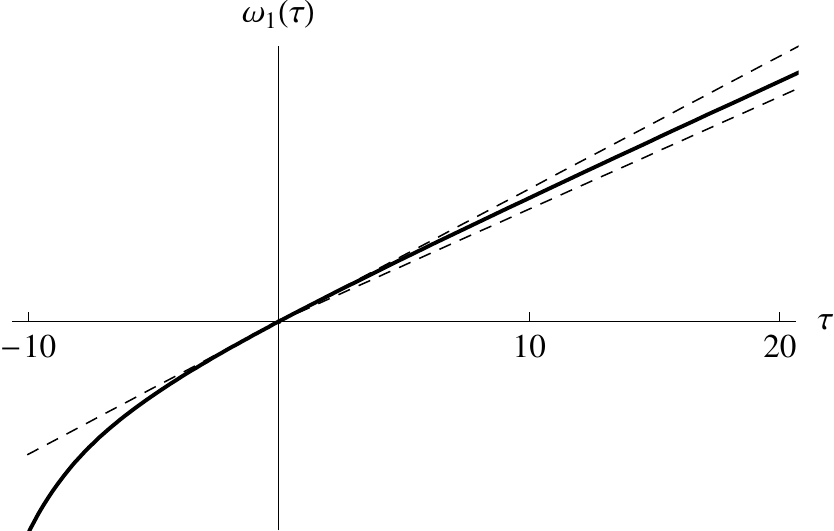}
\caption{The fundamental tone of the disk (solid curve) together with the linear bounds from Lemma~\ref{wbounds}  (dashed lines).}\label{Lemma31bounds}
\end{center}
\end{figure}

Note that Payne \cite{P56} proved linear bounds for eigenvalues of the \emph{clamped} plate under tension. Kawohl, Levine, and Velte \cite{KLV} investigated the sums of the first $d$ eigenvalues as functions of parameters for the clamped plate under tension and compression.

We can also prove another linear upper bound on $\omega_1$, which is just a constant plus the lower bound in Lemma~\ref{wbounds}.
\begin{lemma} \label{wbounds2} For all $\tau\in\RR$,
\[
\omega_1 \leq C(\Omega)+\tau\mu_1,
\]
where the value
\[
 C(\Omega)=\frac{\int_\Omega|D^2v|^2\,dx}{\int_\Omega v^2\,dx}
\]
is given explicitly in terms of the fundamental mode $v$ of the free membrane on $\Omega$.
\end{lemma}
\begin{proof}
 Let $v$ be a fundamental mode of the membrane with $\Delta v=-\mu_1 v$ and $\int_\Omega v\,dx=0$; the membrane boundary condition is $\partial u/\partial n = 0$ on $\partial\Omega$. Then by the variational characterization of eigenvalues,
\begin{align*}
 \omega_1(\tau)\leq Q[v]&= C(\Omega) +\tau Q_\text{M}[v]=C(\Omega)+\tau\mu_1,
\end{align*}
as desired.
\end{proof}

\begin{figure}[h!]
\begin{center}
\includegraphics{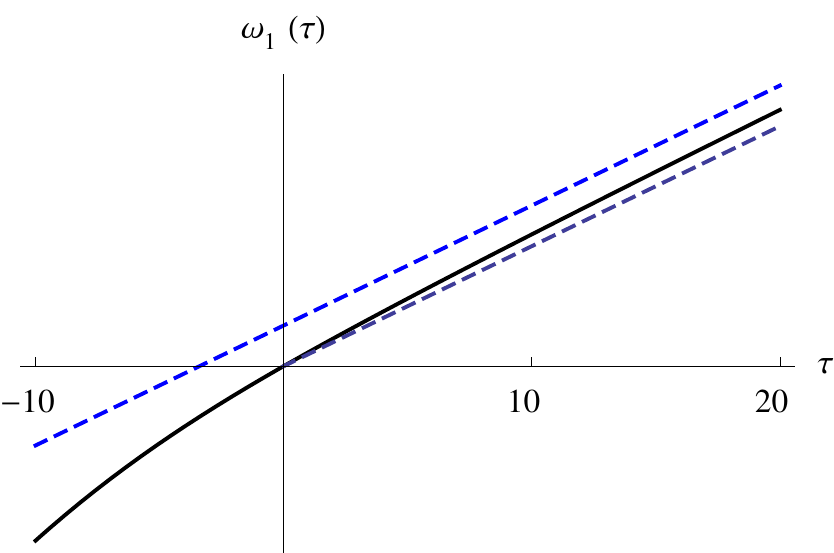}
\caption{The fundamental tone of the disk (solid curve) together with the upper bound of Lemma~\ref{wbounds2} (top dashed line) and the lower bound of Lemma~\ref{wbounds} (bottom dashed line).}
\end{center}
\end{figure}

\subsection*{Infinite tension limit}
A plate behaves like a membrane as the flexural rigidity tends to zero, that is, as $\tau=(\text{tension}/\text{flexural rigidity})$ tends to infinity. For the fundamental tone, that means:
\begin{cor} For the fundamental tone of the free plate,
\[
\frac{\omega_1(\tau)}{\tau}\to\mu_1\qquad\text{as $\tau\to\infty$.}
\]
\end{cor}
\begin{proof}
By Lemmas~\ref{wbounds} and~\ref{wbounds2}, we have
\[
\mu_1\leq \frac{\omega_1(\tau)}{\tau}\leq\mu_1+\frac{C}{\tau}.
\]
Let $\tau\to\infty$.
\end{proof}

The eigenfunctions should converge as $\tau\to\infty$ to the eigenfunctions of the free membrane problem. Proving this for all eigenfunctions seems to require a singular perturbation approach, which has been carried out for the clamped plate in \cite{deGroen}, but we will not need any such facts for our work. For the convergence of the fundamental tone of the clamped plate to the first fixed membrane eigenvalue, see \cite{KLV}.

\begin{rmk} 
In the limit as $\tau\to 0$, we find a relationship between $\omega_1(\tau)/\tau$ and the scalar moment of inertia of the region $\Omega$; see \cite{cthesis}.
\end{rmk}
\section{\bf Summary of Bessel Function facts}
The radial part of the fundamental tone of the unit ball is a linear combination of Bessel functions. Before we begin constructing our trial functions, we need to gather some results established in the companion paper \cite{chasman}.

The ultraspherical Bessel functions $j_l(z)$ of the first kind are defined in terms of the Bessel functions of the first kind, $J_\nu(z)$, as follows:
\begin{align*}
j_{l}(z) &= z^{-s}J_{s+l}(z)\\
\text{with}\quad s &= \frac{d-2}{2}.
\end{align*}
and solve the ultraspherical Bessel equation
\begin{equation}
z^2w''+(d-1)zw'+\Big(z^2-l(l+d-2)\Big)w = 0.\label{besseleqn}.
\end{equation}
Note that this notation suppresses the dependence of the $j_l$ functions on the dimension $d$; we assume dimension $d\geq 2$ is fixed. Ultraspherical modified Bessel functions $i_{l}(z)$ of the first kind are defined analogously, with
\[
i_{l}(z) = z^{-s}I_{s+l}(z)
\]
solving the modified ultraspherical Bessel equation
\begin{equation}
z^2w''+(d-1)zw'-\Big(z^2+l(l+d-2)\Big)w = 0. \label{modbesseleqn}.
\end{equation}

Ultraspherical Bessel functions satisfy the following recurrence relations:
\begin{align}
\frac{d-2+2l}{z}j_l(z) &= j_{l-1}(z)+j_{l+1}(z)\label{j1}\\
j_l^\prime(z) &= \frac{l}{z}j_l(z)-j_{l+1}(z)\label{j2}\\
\frac{d-2+2l}{z}i_l(z) &= i_{l-1}(z)-i_{l+1}(z)\label{i1}\\
i_l^\prime(z) &= \frac{l}{z}i_l(z)+i_{l+1}(z)\label{i2}
\end{align}

Ultraspherical Bessel functions and their derivatives may be expressed by converging series. The first few terms in these expansions may be used to bound the Bessel functions and their derivatives; we will need such bounds for the second derivatives of $j_1$ and $i_1$. Let $d_k$ denote the coefficients of the series expansion for $i_1^{\prime\prime}(z)$, so that
\[
j_1^{\prime\prime}(z) = \sum_{k=1}^\infty (-1)^k d_k z^{2k-1} \quad \text{and} \quad i_1^{\prime\prime}(z) = \sum_{k=1}^\infty d_k z^{2k-1}
\]
by the series expansions of $j_l$ and $i_l$ in \cite[p. 5]{chasman}, where
\begin{align*}
d_k&=\frac{2k+1}{(k-1)!\Gamma(k+1+d/2)}2^{1-2k-d/2}.
\end{align*}

\begin{lemma}\label{ijbounds}\cite[Lemma 10]{chasman} We have the following bounds:
 \begin{align*}
 -d_1 z+d_2 z^3&\geq j_1^{\prime\prime}(z)&\text{for all $z\in\Big[0,\sqrt{3(d+2)/(d+5)}\Big]$,}\\
d_1 z+\frac{6}{5}d_2 z^3&\geq i_1^{\prime\prime}(z)&\text{for all $z\in\Big[0,\sqrt{3}\Big]$.}
 \end{align*}
\end{lemma}
While proofs are provided in the companion paper, these bounds and those listed below in Lemma~\ref{besselsigns} can also all be demonstrated numerically in Mathematica or Maple for any given dimension.

We will also be using additional facts about the signs of certain Bessel functions and derivatives. These were proven in \cite{chasman} and are collected below. We write $\ainf$ for the first nontrivial zero of $j_l^\prime$.

\begin{lemma}\label{besselsigns}\cite[Lemmas 5 through 9]{chasman} We have the following:
 \begin{enumerate}
  \item For $l=1,\dots,5$, we have $j_l>0$ on $(0,\ainf]$.
  \item We have $j_1^\prime>0$ on $(0,\ainf)$.
  \item We have $j_2^\prime>0$ on $(0,\ainf]$.
  \item We have $j_1^{\prime\prime}<0$ on $(0,\ainf]$.
  \item We have $j_l^{(4)}>0$ on $(0,\ainf]$.
 \end{enumerate}
\end{lemma}

We are now ready to begin proving Theorem~\ref{thm1}.

\section{\bf Trial functions}\label{mainthm}
In this and the next sections, we establish the lemmas which allow us to prove Theorem~\ref{thm1}:
\emph{Among all regions $\Omega$ of a fixed volume, when $\tau>0$ the fundamental tone of the free plate is maximal for a ball. That is,}
\begin{equation}
\omega_1(\Omega) \leq \omega_1(\Ostar), \qquad\text{with equality if and only if $\Omega$ is a ball.} \label{c2ineq}
\end{equation}

For simplicity, as we prepare to prove Theorem~\ref{thm1}, we will write $\omega$ instead of $\omega_1$ for the fundamental tone of the free plate with shape $\Omega$; the fundamental tone of the unit ball will be denoted by $\ostar$. When dependence on the region $\Omega$ and the tension $\tau$ need be made explicit, we write $\omega(\tau,\Omega)$ for the fundamental tone and $Q_{\tau,\Omega}$ for the Rayleigh quotient. The tension parameter $\tau>0$ throughout the remainder of this paper.

We begin with the assumption that our domain $\Omega$ has the same volume as the unit ball $\BB$; this is justified by the scaling argument in Lemma~\ref{scaling}.

In \cite[Theorem 2]{chasman}, we identified the fundamental mode of the unit ball, written in spherical coordinates as:
\begin{equation}
 u_1=R(r)Y_1:=\Big(j_1(ar)+\gamma i_1(br)\Big)Y_1,
\end{equation}
where $Y_1$ is any of the $d$ spherical harmonics of order 1, $j_1(z)$ and $i_1(z)$ are ultraspherical Bessel functions, and $a$ and $b$ are positive constants satisfying the conditions $a^2b^2=\omega$ and $b^2-a^2=\tau$. Note that we may take our spherical harmonics $Y_1$ to be the $x_i/r$ for $i=1\dots,d$, where $x_i$ the $i$th coordinate function. 

Finally, recall that we took $\ainf$ to be the first nontrivial zero of $j_l^\prime(z)$; by the proof of \cite[Theorem 3]{chasman}, we have $a<\ainf$. Note that the fundamental tone of the free membrane with shape $\BB$ is given by $\mu_1^*=\ainfsq$.

We are now able to choose our trial functions. Inspired by Weinberger's proof for the membrane \cite{W56}, we choose appropriate trial functions from the fundamental modes of the unit ball. In the following lemmas, we take
\[
R(r)=j_1(ar)+\gamma i_1(br)
\]
to be the radial part of the fundamental mode of the unit ball. Recall $a$ and $b$ are positive constants determined by $\tau$ and the boundary conditions, as in the proof of Theorem~3 in \cite{chasman}. The constant $\gamma$ is positive and determined by $a$, $\tau$, and the boundary conditions as folllows:
\begin{equation}
\gamma :=\frac{-a^2j_1^{\prime\prime}(a)}{b^2i_1^{\prime\prime}(b)}>0.\label{gammadef}
\end{equation}
Recall also that $R(r)>0$ on $(0,1]$ and $R'(1)>0$.

\begin{lemma}\label{trialfcn} (Trial functions) Let the radial function $\rho$ be given by the function $R$, extended linearly. That is,
\[
 \rho(r)=\begin{cases} R(r) &\text{when $0\leq r\leq1$,}\\
       R(1)+(r-1)\Rp(1)&\text{when $r\geq1$.}\\
      \end{cases}
\]
After translating $\Omega$ suitably, the functions $u_k = x_k\rho(r)/r$, for $k=1,\dots,d$, are valid trial functions for the fundamental tone.
\end{lemma}
\begin{proof}
To be valid trial functions, the $u_k$ must be in $H^2(\Omega)$. Because $\Omega$ is bounded in $\RR^d$, the only possible issue would be a singularity at the origin. The series expansions given in \cite[p. 5]{chasman} for $j_1$ and $i_1$ give us that $R(r)/r$ approaches a constant as $r\to 0$. Thus, $u_k\in H^2(\Omega)$ as desired. The trial functions must also be perpendicular to the constant function, and so we will need
\[
\int_\Omega \frac{\rho(r)x_k}{r}\,dx=0\qquad\text{for $k=1,\dots,d$.}
\]
We use the Brouwer Fixed Point Theorem to translate our region so that the above conditions are guaranteed; here again we follow Weinberger \cite{W56}. Write $x=(x_1,\dots,x_d)$ and consider the vector field
\[
 X(v)=\int_\Omega\frac{\rho(|x-v|)}{|x-v|}(x-v)\,dx.
\]
The vector field $X$ is continuous by construction. For any vector $v$ along the boundary of the convex hull of $\Omega$, the vector field $X(v)$ is inward-pointing, because $\rho\geq0$ and the entire region $\Omega$ lies in a half-space to one side of $v$. Thus by the Brouwer Fixed Point Theorem, our vector field $X$ vanishes at some $v$ in the convex hull of $\Omega$. If we first translate $\Omega$ by $v$, then we have $X(0)=\int_\Omega \rho(r)x/r\, dx=0$. This gives us $\int_\Omega u_k\,dx=0$, as desired.
\end{proof}
We will need one further fact about our radial function $\rho$.
\begin{lemma}\label{gppneg} (Concavity)
The function $\rpp(r)\leq0$ for $r\in[0,1]$, with equality only at the endpoints.
\end{lemma}
\begin{proof}
First note that on $[0,1]$, the function $\rho\equiv R$. We see
\[
R^{\prime\prime}(r) = a^2j_1^{\prime\prime}(ar)+\gamma b^2i_1^{\prime\prime}(br),
\]
which is zero at $r=0$ because the individual Bessel derivatives vanish there, by the series expansions for the Bessel $j_l(z)$ and $i_l(z)$ in \cite{chasman}. At $r=1$, the function $R^{\prime\prime}$ vanishes because of the boundary condition $Mu = 0$.

The fourth derivative of $R$ is given by
\[
R^{\prime\prime\prime\prime}(r) = a^4j_1^{\prime\prime\prime\prime}(ar)+\gamma b^4i_1^{\prime\prime\prime\prime}(br).
\]

Because all derivatives of $i_1(z)$ are positive when $z\geq 0$, the second term above is positive on $(0,\infty)$. Lemma~9 in \cite{chasman} states that that $j_1^{\prime\prime\prime\prime}(z)$ is positive on $(0,\ainf]$. Thus $R^{\prime\prime\prime\prime}(r)>0$ on $(0,1]$, and so $R^{\prime\prime}(r)$ is a strictly convex function on $[0,1]$. Since $R^{\prime\prime}=0$ at $r=0$ and $r=1$, the function $R^{\prime\prime}$ must be negative on the interior of the interval $[0,1]$.
\end{proof}

We now bound our fundamental tone above by a quotient of integrals whose integrands are radial functions. The numerator will be quite complicated, so we write
\[
 N[\rho]:=(\rpp)^2+\frac{3(d-1)}{r^4}(\rho-r\rp)^2+\tau(\rp)^2+\frac{\tau(d-1)}{r^2}\rho^2.
\]

We will also need the following calculus facts:
\begin{fact}\cite[Appendix]{cthesis}\label{derivs} We have the sums
\begin{align*}
\sum_{k=1}^d|u_k|^2&=\rho^2\\
\sum_{k=1}^d|Du_k|^2&=\frac{d-1}{r^2}\rho^2+(\rp)^2\\
\sum_{k=1}^d|D^2u_k|^2&= (\rpp)^2+\frac{3(d-1)}{r^4}(\rho-r\rp)^2.
\end{align*}
\end{fact}

We may now use the trial functions to bound our fundamental tone by a quotient of integra;s.
\begin{lemma}\label{lemmaboundRC} (Using the trial functions)
 For any $\Omega$, translated as in Lemma~\ref{trialfcn}, we have
\begin{equation}
\omega \leq \frac{\int_\Omega N[\rho]\,dx}{\int_\Omega \rho^2\,dx}\label{boundRC}
\end{equation}
with equality if $\Omega=\Ostar$.
\end{lemma}

\begin{proof}
For $u_k$ defined as in Lemma~\ref{trialfcn}, we have
\[
\omega \leq Q[u_k] = \frac{\int_\Omega|D^2u_k|^2+\tau|Du_k|^2\,dx}{\int_\Omega|u_k|^2\,dx},
\]
from the Rayleigh-Ritz characterization. We have equality when $\Omega=\Ostar$ because the $u_k$ are the eigenfunctions for the ball associated with the fundamental tone; see \cite{chasman}.  Multiplying both sides by $\int_\Omega|u_k|^2\,dx$ and summing over all $k$, we obtain
\begin{equation}
\omega \int_\Omega \sum_{k=1}^d|u_k|^2\,dx
\leq \int_\Omega\sum_{k=1}^d|D^2u_k|^2+\tau\sum_{k=1}^d|Du_k|^2\,dx\label{multsum}
\end{equation}
again with equality if $\Omega=\Ostar$.

By these sums in Fact~\ref{derivs}, we see inequality \eqref{multsum} becomes
\[
\omega \int_\Omega \rho^2\,dx \leq \int_\Omega\left((\rpp)^2+\frac{3(d-1)}{r^4}(\rho-r\rp)^2+\tau(\rp)^2+\frac{\tau(d-1)}{r^2}\rho^2\right)\,dx,
\]
once more with equality if $\Omega$ is the ball $\Ostar$. Dividing both sides by $\int_\Omega \rho^2\,dx$, we obtain \eqref{boundRC}.
\end{proof}

\section{\bf Partial monotonicity of the integrands}
We want to show the quotient \eqref{boundRC} in Lemma~\ref{lemmaboundRC} has a sort of monotonicity with respect to the region $\Omega$, and so we examine the integrands of the numerator and denominator separately. The case of the denominator is much simpler; the partial monotonicity of the integrand of the numerator is much more difficult, and requires several lemmas.

We begin with the denominator.
\begin{lemma}\label{mondenom}(Monotonicity in the denominator)
 The function $\rho(r)^2$ is strictly increasing.
\end{lemma}
\begin{proof}
Differentiating, we see
\[
 \rp(r)=\begin{cases} j_1^\prime(ar)+\gamma i_1^\prime(br) &\text{when $0\leq r \leq 1$,}\\
         R^\prime(1) &\text{when $r\geq1$.}
        \end{cases}
\]
Obviously $i_1^\prime(br)\geq 0$. Because we have $a<\ainf$ from the proof of Theorem~3 in \cite{chasman}, the function $j_1^\prime(ar)$ is positive on $[0,1]$. Thus $\rp(r)$ is positive everywhere, and $\rho$ (and therefore $\rho^2$) is an increasing function.
\end{proof}

We do not need to prove the integrand of the numerator is strictly decreasing; a weaker ``partial monotonicity" condition is sufficient. We will say a function $F$ is \emph{partially monotonic for $\Omega$} if it satisfies
\begin{equation}
F(x)> F(y) \qquad\text{for all $x\in\Omega$ and $y\not\in\Omega$.}\label{moncond}
\end{equation}

\begin{lemma}\label{monnum}(Partial monotonicity in the numerator)
The function
\[
N[\rho]=(\rpp)^2+\frac{3(d-1)}{r^4}(\rho-r\rp)^2+\tau\left((\rp)^2+(d-1)\frac{\rho^2}{r^2}\right)
\]
satisfies condition \eqref{moncond} for $\Omega$ the unit ball.
\end{lemma}
\begin{proof}
Given that $\rpp < 0$ on $(0,1)$ and equals zero elsewhere by Lemma~\ref{gppneg}, the function $(\rpp)^2$ satisfies condition \eqref{moncond} for the unit ball. The derivative of the function $\tau(\rp)^2$ with respect to $r$ is $2\tau\rp\rpp$, and hence negative on $(0,r)$ and zero everywhere else. Thus $\tau(\rp)^2$ is a decreasing function of $r$. It remains to show that the remaining term
\[
h(r) =\frac{3(\rho-r\rp)^2}{r^4}+\tau\frac{\rho^2}{r^2}
\]
is also a decreasing function of $r$. Differentiating, we see
\begin{align*}
h'(r) &=\frac{-2}{r^3}(\rho-r\rp)\left(\frac{6}{r^2}(\rho-r\rp)+3\rpp+\tau \rho\right).
\end{align*}
Now, $\rho-r\rp=0$ at $r=0$ and
\[
 \frac{d}{dr} (\rho-r\rp) = -r\rpp,
\]
so by Lemma~\ref{gppneg}, $(\rho-r\rp)$ is positive on $(0,\infty)$ and vanishes at zero. Thus in order for $h(r)$ to be decreasing, we must have
\begin{equation}
  \frac{6}{r^2}(\rho-r\rp)+3\rpp+\tau \rho>0. \label{quantinterest}
\end{equation}

Let $\Delta_r\rho := \rpp-(d-1)r^{-2}(\rho-r\rp)$. Recall from the Bessel equations~\eqref{besseleqn} and~\eqref{modbesseleqn} that
\begin{equation}
\Delta_rj_1(ar)=-a^2j_1(ar)\quad\text{and}\quad\Delta_ri_1(br)=b^2i_1(br).\label{delta}
\end{equation}
Then on the interval $[0,1]$,
\begin{align*}
\frac{6}{r^2}(\rho-r\rp)+3\rpp+\tau\rho &=\frac{6}{r^2}(\rho-r\rp)+3\left(\Delta_r\rho+\frac{d-1}{r^2}(\rho-r\rp)\right) +\tau\rho\\
&=\frac{3(d+1)}{r^2}(\rho-r\rp)+3\Big(-a^2j_1(ar)+\gamma b^2i_1(br)\Big)+\tau\rho,
\end{align*}
with the last equality by \eqref{delta}. Considering the first term of the last line above, we see by \eqref{j2} and \eqref{i2},
\begin{align*}
\frac{1}{r^2}(\rho-r\rp)&=\frac{1}{r^2}\Big(arj_2(ar)-br\gamma\, i_2(br)\Big)\\
&=\frac{1}{d+2}\left(a^2\Big(j_1(ar)+j_3(ar)\Big)+\gamma b^2\Big(i_3(br)-i_1(br)\Big)\right)
\end{align*}
with the second equality by \eqref{j1}, \eqref{i1}, and simplifying.

Therefore our quantity of interest in \eqref{quantinterest} can be bounded below in terms of $j_l$'s and $i_l$'s:
\begin{align*}
\frac{6}{r^2}(\rho-r\rp)&+3\rpp+\tau\rho \\
&=\left(\tau-\frac{3a^2}{d+2}\right)j_1(ar)+\frac{3a^2(d+1)}{d+2}j_3(ar)\\
&\quad+\gamma\left(\tau+\frac{3b^2}{d+2}\right)i_1(br)+\gamma\frac{3b^2(d+1)}{d+2}i_3(br).\\
&\geq \frac{3a^2(d+1)}{d+2}j_3(ar)+\gamma\frac{3b^2(d+1)}{d+2}i_3(br)\\
&\quad+\left(\tau-\frac{3a^2}{d+2}+\gamma\left(\tau+\frac{3b^2}{d+2}\right)\right)j_1(ar)
\end{align*}
with the inequality by $j_l(ar) \leq i_l(ar)\leq i_l(br)$, since $\tau>0$ and so $a<b$.

The function $i_3$ is everywhere nonnegative and the constant $\gamma$ is positive. We have $a<\ainf$, so $ar<\ainf$ on $[0,1]$ and hence the functions $j_3(ar)$ and $j_1(ar)$ are positive on $[0,1]$ by Lemma~1 of \cite{chasman}. The remaining factor is positive for all $\tau>0$ by Lemmas~\ref{largetau} and \ref{smalltau} (to follow), thus establishing \eqref{quantinterest} and completing the proof.
\end{proof}

We establish the positivity of the remaining factor first for those $\tau $ values such that $\tau>9/(d+5)$; the proof for smaller $\tau$ values is more complicated and is treated in another lemma.
\begin{lemma}\label{largetau} (Large $\tau$) We have
\begin{equation}
\tau-\frac{3a^2}{d+2} > 0  \label{tclem}
\end{equation}
for all $\tau>9/(d+5)$.
\end{lemma}
\begin{proof} We use the bounds we established for $\omega(\tau)$ in Section~\ref{spectrum}.

Recall that the first free membrane eigenvalue for the ball is $\mu_1^*=\ainfsq$. Lemma~\ref{wbounds} and Proposition of Lorch and Szego (\cite{LS94}, but see the statement in \cite[Prop 4]{chasman}) together give $(d+2) \tau > \ostar > \tau d$. Because $\ostar = a^4+a^2\tau$ \cite[Prop 2]{chasman}, we obtain inequalities relating $\tau$ and $a$:
\begin{equation}
\frac{a^4}{d-a^2}> \tau > \frac{a^4}{d+2-a^2}, \label{abounds}
\end{equation}
with the upper bound holding only if $a^2<d$.

Using the lower bound, we see
\begin{align*}
\tau-\frac{3a^2}{d+2}&>\frac{a^4}{d+2-a^2}-\frac{3a^2}{d+2}\\
&= \frac{a^4(d+5)-3a^2(d+2)}{(d+2)(d+2-a^2)}
\end{align*}
which is nonnegative whenever $a^2 \geq 3(d+2)/(d+5)$. When $a^2 <3(d+2)/(d+5)$, we have
\[
\tau - \frac{3a^2}{d+2} >\tau- \frac{9}{d+5}>0,
\]
by our choice of $\tau$.
\end{proof}

\begin{lemma}\label{smalltau} (Small $\tau$) We have
\begin{equation}
 \tau-\frac{3a^2}{d+2}+\gamma\left(\tau+\frac{3b^2}{d+2}\right) > 0 \label{smalltaucond}
\end{equation}
for all $0<\tau\leq9/(d+5)$.
\end{lemma}

\begin{proof} The proof will proceed as follows. For $0<\tau\leq9/(d+5)$, we restate the desired inequality \eqref{smalltaucond} as a condition on $\gamma$, \eqref{gammacond}. We then use properties of Bessel functions to establish a lower bound on $\gamma$ in terms of a rational function of $a$; we then show this function satisfies \eqref{gammacond}. We will need to treat the cases of $d\geq 3$ and $d=2$ separately, because the two-dimensional case requires better bounds than we can derive for general $d$.

First note that $b^2=a^2+\tau$, so the inequality \eqref{smalltaucond} is equivalent to
\begin{equation}
\tau > \frac{3a^2(1-\gamma)}{(d+5)\gamma+(d+2)}.
\end{equation}
Using the lower bound on $\tau$ in \eqref{abounds}, we see that the above will hold if
\begin{equation}
\gamma \geq \frac{3(d+2)-a^2(d+5)}{(3+a^2)(d+2)}=:\gamma^*.\label{gammacond}
\end{equation}

We need only show that \eqref{gammacond} holds for all $0<\tau \leq 9/(d+5)$. We will use Taylor polynomial estimates to bound $\gamma$ below by a rational function. From Lemma~6 of \cite{chasman}, we have
\begin{align*}
 j_1''(z)&\leq -d_1z+d_2z^3 &\text{on $[0,\sqrt{3(d+2)/(d+5)}]$,}\\
i_1''(z) &\leq d_1z+\frac{6}{5}d_2z^3&\text{on $[0,\sqrt{3}]$,}.
\end{align*}
These bounds apply to $z=ar$ and $z=br$ respectively, when $r\in[0,1]$, as we show below by obtaining bounds on $a^2$ and $b^2$.

To derive our bound on $a^2$, we note that the lower bound of \eqref{abounds} together with our assumption $\tau\leq 9/(d+5)$ implies
\[
 \frac{a^4}{d+2-a^2}<\frac{9}{d+5},
\]
so that 
\[
(d+5)a^4+9a^2-9(d+2)<0.
\]
The left-hand side is increasing with respect to $a^2$ and equals zero when $a^2=3(d+2)/(d+5)$. Hence $a^2<3(d+2)/(d+5)$ and the bound on $j_1''(z)$ holds for $z=ar$ when $\tau\leq 9/(d+5)$. We use these to obtain a further bound:
\begin{align*}
0&\geq\tau-\frac{9}{d+5}\\
&>\frac{a^4}{d+2-a^2}-\frac{9}{d+5}\\
&=\frac{a^2+3}{d+2-a^2}\left(a^2-\frac{3(d+2)}{d+5}\right),
\end{align*}
and so we have $a^2<d$.

To bound $b^2$, we use $b^2=a^2+\tau$ and obtain
\[
b^2=a^2+\tau \leq \frac{3(d+2)}{d+5}+\frac{9}{d+5} = 3,
\]
and so $b^2\leq 3$.

We also have, from \eqref{abounds},
\begin{equation}
\frac{da^2}{d-a^2} > b^2 > \frac{(d+2)a^2}{d+2-a^2},\label{bbounds}
\end{equation}
with the upper bound holding in this regime because $a^2<d$.

We also need the following binomial estimate:
\begin{equation}
 1-\frac{3}{2}x<(1-x)^{3/2} \qquad\text{for $0<x<1$.}\label{32bound}
\end{equation}

Using these bounds, we see
\begin{align*}
\gamma &= \frac{-a^2 j_1^{\prime\prime}(a)}{b^2 i_1^{\prime\prime}(b)} &&\text{by definition \eqref{gammadef}}\\
&\geq \frac{a^2(d_1a-d_2a^3)}{b^2(d_1b+(6/5)d_2b^3)}&&\text{by Lemma~6 of \cite{chasman}}\\
&\geq \frac{a^3(d_1-d_2a^2)}{\left(\frac{da^2}{d-a^2}\right)^{3/2}(d_1+(6/5) d_2 \frac{da^2}{d-a^2})} &&\text{by \eqref{bbounds}}\\
&= \left(\frac{d-a^2}{d}\right)^{3/2}\frac{(d-a^2)(1-c_1a^2)}{(d-a^2+(6/5) c_1da^2)}&&\text{writing $c_1=d_2/d_1=\frac{5}{6(d+4)}$}\\
&\geq \left(1-\frac{3a^2}{2d}\right)\frac{(d-a^2)(6(d+4)-5a^2)}{(6d(d+4)-24a^2)} &&\text{by \eqref{32bound},}
\end{align*}
noting that $a^2/d < 1$ and $a^2<3(d+2)/(d+5)$.

Thus we have $\gamma-\gamma^*\geq 0$ if
\[
\left(1-\frac{3a^2}{2d}\right)\frac{(d-a^2)(6(d+4)-5a^2)}{6d(d+4)-24a^2}-\frac{3(d+2)-a^2(d+5)}{(3+a^2)(d+2)}\geq 0,
\]
or, clearing the denominators and writing $x=a^2$, if
\[
(2d-3x)(d-x)\Big(6(d+4)-5x\Big)(3+x)(d+2)-2d\Big(6d(d+4)-24x\Big)\Big(3(d+2)-x(d+5)\Big)\geq 0.
\]
The above polynomial is fourth degree in each of $d$ and $x$ and has the root $x=0$; because we are only interested in its behavior for $x\in(0,3(d+2)/(d+5))$, we may divide by $x$ and work to show the resulting polynomial 
\begin{align*}
P(x,d) &= 24d^4 +60d^3 -120d^2-432d -40 d^3 x -119d^2x -6dx +432x \\
&\qquad+43d^2x^2+113dx^2+54x^2-15dx^3-30x^3
\end{align*}
is nonnegative for $x\in(0,3(d+2)/(d+5)$. This claim is addressed in Lemma~\ref{poly1} for $d\geq3$.

For $d=2$, the function $P(x,2)$ is negative on most of our interval of interest $[0,12/7]$, and so we must improve our lower bound on $\gamma$. The derivation follows that of inequality \eqref{abounds} in the proof of Lemma~\ref{largetau}, as follows. 

By Lemma~\ref{wbounds}, $\ostar>\ainfsq\tau$, where $\ainf\approx1.84118$ is the first zero of $J_1(z)$.  By Proposition~2 of \cite{chasman} we have $\ostar=a^4+a^2\tau$, giving us
\[
\tau \leq \frac{a^4}{\ainfsq- a^2}
\]
Using $b^2=a^2+\tau$, we obtain also a bound on $b^2$:
\[
\quad b^2\leq \frac{\ainfsq a^2}{\ainfsq-a^2}.
\]
Proceeding as before, we deduce
\[
\gamma \geq \left(1-\frac{3a^2}{2\ainfsq}\right)\frac{(\ainfsq-a^2)(36-5a^2)}{36\ainfsq+(6\ainfsq-36)a^2}
\]
with the last again from \eqref{32bound}. So $\gamma-\gamma^*\geq 0$ if
\[
\left(1-\frac{3a^2}{2\ainfsq}\right)\frac{(\ainfsq-a^2)(36-5a^2)}{36\ainfsq+(6\ainfsq-36)a^2}-\frac{12-7a^2}{12+4a^2}\geq 0
\]
or, setting $x=a^2$, if the fourth degree polynomial
\[
Q(x)=\left(1-\frac{3x}{2\ainfsq}\right)(\ainfsq-x)(36-5x)(12+4x)-\Big(36\ainfsq+(6\ainfsq-36)x\Big)(12-7x)
\]
is positive on $[0,12/7]$. This positivity follows from Lemma~\ref{poly2}, completing our proof.
\end{proof}

The next two lemmas regarding the polynomials $P$ and $Q$ allow us to complete the proof of Lemma~\ref{smalltau}.

\begin{lemma}\label{poly1} The polynomial
\begin{align*}
P(x,d) &= 24d^4 +60d^3 -120d^2-432d -40 d^3 x -119d^2x -6dx +432x \\
&\qquad+43d^2x^2+113dx^2+54x^2-15dx^3-30x^3
\end{align*}
is nonnegative for all $x\in(0,3(d+2)/(d+5))$ and integers $d\geq 3$.
\end{lemma}
\begin{proof}
First note that $3(d+2)/(d+5)<3$. We bound $P$ below on the interval $x\in[0,3]$ by taking $x=3$ in terms with negative coefficients and taking $x=0$ in terms with positive coefficients, obtaining
\[
P(x,d)\geq  24d^4 -60d^3-477d^2 -855 d - 810 =: g(d).
\]
The highest order term is $24d^4$, and so $g$ is ultimately positive and increasing in $d$. Note also that
\begin{align*}
g'(d)&= 96 d^3- 180 d^2- 954 d-855\\
g''(d) &= 288d^2-360d-954.
\end{align*}
The function $g''(d)$ is a quadratic polynomial with positive leading coefficient and roots at $d \approx -1.30$ and $2.55$; thus $g'(d)$ is increasing for all $d\geq 3$.  We see that $g'(5)=1875$, so $g$ is increasing for all $d\geq 5$. Finally, $g(7)=6876$, so for all $d\geq7$ we have $g(d)> 0$ and hence $P(x,d)>0$ for all $d\geq 7$ and $x\in[0,3]$.

For $d=3,4,5,6$, we look at the polynomials $P_d(x)=P(x,d)$ directly to show that $P_d(x)>0$ on $[0,3(d+2)/(d+5)]$.  Each $P_d$ is a cubic polynomial in $x$; its first derivative $P_d^\prime(x)$ is quadratic and so the critical points of $P_d(x)$ can all be found exactly.

For $d=4$, $5$, and $6$, direct calculations show $P'_d<0$ on $[0,3(d+2)/(d+5)]$ and $P_d(3(d+2)/(d+5))>0$, so $P_d(x)>0$ on $[0,3(d+2)/(d+5)]$.

For $d=3$, our interval of interest is $[0,15/8]$. We have a critical point \\ $c\approx1.4\in[0,15/8]$, with $P_3'<0$ on $[0,c]$ and $P_3'>0$ on $[c,15/8]$. The critical value $P_3(c)\approx 79$ is positive, so $P_3(x)>0$ on the desired interval $[0,15/8]$.
\end{proof}

\begin{lemma}\label{poly2} The polynomial
\[
Q(x)=\left(1-\frac{3x}{2\ainfsq}\right)(\ainfsq-x)(36-5x)(12+4x)-\Big(36\ainfsq+(6\ainfsq-36)x\Big)(12-7x)
\]
is positive on $[0,12/7]$.
\end{lemma}
\begin{proof}
As in previous cases, $x=0$ is a root of this polynomial, so we examine $g(x):=Q(x)/x$. The derivative $g'(x)$ is a quadratic polynomial, so its roots can be found exactly. We see that $g$ has a critical point $c\approx 1.4$ in $[0,12/7]$, with $g'<0$ on $[0,c]$ and $g'>0$ on $[c,12/7]$. The critical value $g(c)\approx 177.8$ is positive, so $g>0$ on $[0,12/7]$.
\end{proof}

\section{\bf Proof of the isoperimetric inequality}
Now that we have established the desired monotonicity of our quotient, we need two more lemmas before we can prove the isoperimetric inequality for the free plate under tension. Our next lemma is a simple observation about integrals of monotone and partially monotone functions, which is a special case of more general rearrangement inequalities (see \cite[Chapter 3]{liebloss}).
\begin{lemma}\label{monint}
For any radial function function $F(r)$ that satisfies the partial monotonicity condition \eqref{moncond} for $\Ostar$,
\[
\int_\Omega F\,dx \leq \int_{\Ostar}F\,dx
\]
with equality if and only if $\Omega=\Ostar$.
For any strictly increasing radial function $F(r)$,
\[
\int_\Omega F\,dx \geq \int_{\Ostar}F\,dx
\]
with equality if and only $\Omega=\Ostar$.
\end{lemma}
\begin{proof} Note that $|\Omega|=|\Ostar|$ with $|\Omega\backslash\Ostar|=|\Ostar\backslash\Omega|$. Suppose $F$ satisfies \eqref{moncond} for $\Ostar$. The result follows from decomposing the domain:
\begin{align*}
\int_\Omega F\,dx &=\int_{\Omega\cap\Ostar}F\,dx+\int_{\Omega\setminus\Ostar}F\,dx\\
&\leq \int_{\Omega\cap\Ostar}F\,dx+\sup_{x\in\Omega\setminus\Ostar}|F(x)||\Omega\setminus\Ostar|  \\
&\leq \int_{\Omega\cap\Ostar}F\,dx+\inf_{x\in\Ostar\setminus\Omega}|F(x)||\Ostar\setminus\Omega|&&\text{since $F$ satisfies \eqref{moncond}.}\\
&\leq \int_{\Omega\cap\Ostar}F\,dx+\int_{\Ostar\setminus\Omega}F\,dx\\
&=\int_\Ostar F\,dx.
\end{align*}
Note that if $|\Omega\backslash\Ostar|>0$, either the second inequality or the third is strict by the strict inequality in \eqref{moncond}. If $F$ is  increasing, then apply the first part of the Lemma to the function $-F$.
\end{proof}

The final lemma describes how the eigenvalues change with the dilation of the region, and is used in the proof of the theorem to show we need only consider $\Omega$ with volume equal to that of the unit ball. We will use the notation $s\Omega := \{x\in\RR^d:x/s\in\Omega\}$ for $s>0$.
\begin{lemma}\label{scaling} (Scaling) For all $s>0$, we have
\[
 \omega(\tau,\Omega)=s^{4}\omega(s^{-2}\tau,s\Omega).
\]
\end{lemma}
\begin{proof}
For any $u\in H^2(\Omega)$ with $\int_\Omega u\,dx=0$, let $\tilde{u}(x)=u(x/s)$. Then $\tilde{u}$ is a valid trial function on $s\Omega$ and so
\begin{align*}
Q_{s^{-2}\tau,s\Omega}[\tilde{u}]&=\frac{\int_{s\Omega}|D^2\tilde{u}|^2+s^{-2}\tau|D\tilde{u}|^2\,dx}{\int_{s\Omega}\tilde{u}^2\,dx}\\
&=\frac{\int_{s\Omega}|s^{-2}(D^2u)(x/s)|^2+s^{-2}\tau|s^{-1}(Du)(x/s)|^2\,dx}{\int_{s\Omega}u(x/s)^2\,dx}\\
&=\frac{s^{-4+d}\int_\Omega|D^2u|^2+\tau|Du|^2\,dy}{s^{d}\int_\Omega u^2\,dy} \qquad\qquad\qquad\text{taking $y=x/s$,}\\
&=s^{-4}Q_{\tau,\Omega}[u].
\end{align*}
Now the lemma follows from the variational characterization of the fundamental tone.
\end{proof}

We can now prove our main result.
\begin{proof}[Proof of Theorem~\ref{thm1}]
Once we have established inequality \eqref{c2ineq} for all regions $\Omega$ of volume equal to that of the unit ball and all $\tau>0$, we obtain \eqref{c2ineq} for regions of arbitrary volume, since
\[
\omega(\tau,\Omega)=s^{4}\omega(s^{-2}\tau,s\Omega)\leq s^{4}\omega(s^{-2}\tau,s\Ostar)=\omega(\tau,\Ostar),
\]
for all $s>0$ by Lemma~\ref{scaling}.

Thus it suffices to prove the theorem for $\Omega$ with volume equal to that of the unit ball, so that $\Ostar$ is the unit ball. We may also translate $\Omega$ as in Lemma~\ref{trialfcn}, which leaves the fundamental tone unchanged. Then,
\begin{align*}
\omega &\leq \frac{\int_\Omega N[\rho]\,dx}{\int_\Omega \rho^2\,dx} &&\text{by Lemma~\ref{lemmaboundRC}}\\
&\leq \frac{\int_\Ostar N[\rho]\,dx}{\int_\Ostar \rho^2\,dx}&&\text{by Lemmas~\ref{mondenom}, \ref{monnum}, and~\ref{monint}}\\
&=\ostar,
\end{align*}
by applying the equality condition in Lemma~\ref{lemmaboundRC}. Finally, if equality holds, then $\Omega$ must be a ball, by the equality statement in Lemma~\ref{monint}.
\end{proof}

\section{\bf Further Directions}
The isperimetric problem for the free plate considered in this paper can be generalized in several different directions: considering the case where the material property Poisson's ratio is nonzero, investigating a stronger inequality involving the harmonic mean of eigenvalues, and considering the problem on curved spaces.
\subsection*{Poisson's Ratio}
One generalization of the free plate problem is to account for Poisson's Ratio, a property of the material of the plate that describes how a rectangle of the material stretches or shrinks in one direction when stretched along the perpendicular direction. Our Rayleigh quotient and work so far all hold for a material where Poisson's Ratio is zero. Most real-world materials have $\sigma\in[0,1/2]$, although there exist some materials with negative Poisson's Ratio. 

We will assume $\sigma \in [0,1)$ in order to be assured of coercivity of the generalized Rayleigh quotient, given by
\begin{equation}\label{RQsigma}
Q[u] = \frac{\int_\Omega (1-\sigma) |D^2 u|^2 + \sigma(\Delta u)^2 + \tau |D u|^2\,dx}{\int_\Omega |u|^2\,dx}.
\end{equation}
This quotient reduces to our previous quotient \eqref{RQ} when $\sigma=0$. Following our earlier derivation, we obtain the same eigenvalue equation
\begin{equation*}
\Delta \Delta u - \tau \Delta u = \omega u,
\end{equation*}
along with new natural boundary conditions on $\partial\Omega$, which reduce to the old ones when $\sigma=0$.
\begin{align*}
(1-\sigma)\frac{\partial^2u}{\partial n^2}+\sigma\Delta u\Big|_{\partial\Omega} = 0,\\
\tau\frac{\partial u}{\partial n}-(1-\sigma)\sdiv\left(\sproj\left[(D^2u)n\right]\right)-\frac{\partial\Delta u}{\partial n}\Big|_{\partial\Omega} = 0.
\end{align*}

The generalization to nonzero $\sigma$ does not change the eigenvalue equation and hence the general form of solutions is preserved. However, the change in the Rayleigh quotient affects the proof of Theorem~3 of \cite{chasman}, which identified the fundamental mode of the ball.  This in turn affects the proof of the isoperimetric inequality in this paper. We can no longer complete the square in the Rayleigh quotient as in \cite[Theorem 3]{chasman} to show the fundamental mode of the ball corresponds to $l=0$ or $1$, although for some values of $\sigma$ we can adapt the proof to show the lowest eigenvalue corresponding to $l=1$ is lower than that for $l=0$.

\subsection*{Harmonic mean of low eigenvalues}
In two dimensions, Szeg\H o was able to prove a stronger statement of the Szeg\H o-Weinberger inequality using conformal mappings \cite{S50,Serrata}. Specifically, he proved that the sum of reciprocals
\[
\frac{1}{\mu_1}+\frac{1}{\mu_2}
\]
is minimal for a disk. In other words, the harmonic mean of $\mu_1$ and $\mu_2$ is maximal for the disk.  Our investigation in \cite[Chapter 3]{cthesis} with the moment of inertia suggests a similar result for the free plate, since the moment of inertia is minimal for a ball. That is, for the free plate, we conjecture
\[
\frac{1}{d}\sum_{i=1}^d\frac{1}{\omega_i(\Omega)}\geq \frac{1}{\omega_1(\Ostar)}.
\]

\subsection*{Curved spaces}
We have taken our region $\Omega$ to be in Euclidean space $\RR^d$, but we could consider the same eigenvalue problem on a region in spaces of constant curvature: the sphere and hyperbolic space. Other eigenvalue inequalities have been proven in these spaces \cite{Asummary}. In particular, the Szeg\H o-Weinberger inequality was proved for domains on the sphere by Ashbaugh and Benguria \cite{ABSW}. Another direction of generalization would be Hersch-type bounds for metrics on the whole sphere or torus; see \cite{hersch}.

\section*{Acknowledgments} I am grateful to the University of Illinois
Department of Mathematics and the Research Board for support during my
graduate studies, and the National Science Foundation for graduate student support
under grants DMS-0140481 (Laugesen) and DMS-0803120 (Hundertmark) and DMS 99-83160
(VIGRE), and the University of Illinois Department of Mathematics for travel support to attend the 2007 Sectional meeting of the AMS in New York. I would also like to thank the Mathematisches Forschungsinstitut Oberwolfach for travel support to attend the workshop on Low Eigenvalues of Laplace and Schr\"{o}dinger Operators in 2009. Finally, I would like to thank my advisor Richard Laugesen for his support and guidance throughout my time as his student and for his assistance with refining this paper.


\begin{thebibliography}{9}

\bibitem{Asummary}
M. S. Ashbaugh. \emph{Isoperimetric and universal inequalities for eigenvalues.} Spectral theory and geometry (Edinburgh, 1998), 95--139, London Math. Soc. Lecture Note Ser., 273, Cambridge Univ. Press, Cambridge, 1999.

\bibitem{ABSW}
M. S. Ashbaugh and R. Benguria. \emph{Sharp upper bound to the first nonzero Neumann eigenvalue for bounded domains in spaces of constant curvature}. J. London Math. Soc. (2) \textbf{52} (1995), no. 2, 402--416. 

\bibitem{AB95}
M. S. Ashbaugh and R. Benguria. \emph{On Rayleigh's conjecture for the clamped plate and its generalization to three dimensions}, Duke Math. J., \textbf{78} (1995), 1--17.

\bibitem{AL96}
M. S. Ashbaugh and R. S. Laugesen. \emph{Fundamental tones and buckling loads of clamped plates.}
Ann. Scuola Norm. Sup. Pisa Cl. Sci. (4) \textbf{23} (1996), no. 2, 383--402.

\bibitem{Bandle}
C. Bandle.  \textrm{Isoperimetric Inequalities and Applications}, Pitman Advances Publishing Program, Boston, London, Melbourne, 1980.

\bibitem{chasman}
L. M. Chasman. \emph{The fundamental tone of the free circular plate.} Preprint. arXiv:1004.3316 [math.AP]

\bibitem{cthesis}
L. M. Chasman. \emph{Isoperimetric problem for eigenvalues of free plates.} Ph.D thesis, University of Illinois at Urbana-Champaign, 2009. arXiv:1004.0016 [math.SP]

\bibitem{deGroen}
P. P. N. de Groen. \emph{Singular perturbations of spectra}. Asymptotic analysis,
Lecture Notes in Math., \textbf{711}, Springer, Berlin, 1979, 9--32.

\bibitem{GT}
D. Gilbarg and N. S. Trudinger. \textrm{Elliptic Partial Differential Equations of Second Order}. Springer-Verlag, Berlin, 2001. (Reprint of 1998 edition.)

\bibitem{H06}
A. Henrot. \textrm{Extremum problems for eigenvalues of elliptic operators.} Frontiers in Mathematics. Birkh\"auser Verlag, Basel, 2006.

\bibitem{KLV}
B. Kawohl, H. A. Levine and W. Velte. \emph{Buckling eigenvalues for a clamped plate embedded in an elastic medium and related questions.}  SIAM J. Math. Anal.  \textbf{24}  (1993),  no. 2, 327--340.

\bibitem{Kesavan}
S. Kesavan. \textrm{Symmetrization and Applications.} World Scientific, Singapore, 2006.

\bibitem{KS52}
E.T. Kornhauser and I. Stakgold, \emph{A variational theorem for $\nabla^2u+\lambda u=0$ and its application}. J. Math. and Phys. \textbf{31} (1952), 45--54.

\bibitem{hersch}
J. Hersch. \emph{Quatre propri\'et\'es isop\'erim\'etriques de membranes sph\'eriques homog\`enes}. C. R. Acad. Sci. Paris S�r. A-B \textbf{270} (1970), A1645--A1648. 

\bibitem{LW73}
H. P. Licari and H. Warner. \emph{Domain dependence of eigenvalues of vibrating plates.} SIAM J. Appl. Math. \textbf{24}, No 3, (1973), 383--395.

\bibitem{liebloss}
E. H. Lieb and M. Loss. \textrm{Analysis}. Second edition. Graduate Studies in Mathematics, 14. American Mathematical Society, Providence, RI, 2001.

\bibitem{LS94}
L. Lorch and P. Szego. \emph{Bounds and monotonicities for the zeros of derivatives of ultraspherical Bessel functions.} SIAM J. Math. Anal. \textbf{25} (1994), no. 2, 549--554.

\bibitem{N92}
N. S. Nadirashvili. \emph{New isoperimetric inequalities in mathematical physics}. Partial differential equations of elliptic type (Cortona, 1992), 197--203, Sympos. Math. XXXV, Cambridge Univ. Press, Cambridge, 1994

\bibitem{N95}
N. S. Nadirashvili. \emph{Rayleigh's conjecture on the principal frequency of the clamped plate}, Arch. Rational Mech. Anal., \textbf{129} (1995), 1--10.

\bibitem{NS07}
S.A. Nazarov and G. Sweers. \emph{A hinged plate equation and iterated Dirichlet Laplace operator on domains with concave corners.} J. Differential Equations \textbf{233}(1), (2007), 151--180. 

\bibitem{Nir55}
L. Nirenberg. \emph{Remarks on strongly elliptic partial differential equations.} Communications in Pure and Applied Mathematics \textbf{8} (1955), 649--675.

\bibitem{P56}
L. E. Payne. \emph{New isoperimetric inequalities for eigenvalues and other physical quantities.} Comm. Pure Appl. Math., \textbf{9}, (1956), 531--542.

\bibitem{P58}
L. E. Payne. \emph{Inequalities for eigenvalues of supported and free plates.} Quart. Appl. Math. \textbf{16}, (1958), 111--120.

\bibitem{RToS}
J.W.S Rayleigh. \textrm{The theory of sound}, Dover Pub, New York, 1945. Re-publication of the 1894/96 edition.

\bibitem{SH77}
R. E. Showalter. \textrm{Hilbert space methods for partial differential equations.} Monographs and Studies in Mathematics, Vol. 1. Pitman, London-San Francisco, Calif.-Melbourne, 1977.

\bibitem{S50}
G. Szeg\H o. \emph{On membranes and plates.} Proc. Nat. Acad. Sci., \textbf{36} (1950), 210--216.

\bibitem{S54}
G. Szeg\H o. \emph{Inequalities for certain eigenvalues of a membrane of given area.}
J. Rational Mech. Anal. \textbf{3}, (1954), 343--356.

\bibitem{Serrata}
G. Szeg\H o. \emph{Note to my paper ``On membranes and plates''.} Proc. Nat. Acad. Sci. (USA) \textbf{44} (1958), 314--316.

\bibitem{T81}
G. Talenti. \emph{On the first eigenvalue of the clamped plate}. Ann. Mat. Pura Appl. (Ser. 4), \textbf{129} (1981), 265--280.

\bibitem{taylor}
M. E. Taylor. \textrm{Partial Differential Equations. I. Basic Theory}. Applied Mathematical Sciences, 115. Springer-Verlag, New York, 1996.

\bibitem{verchota}
G. C. Verchota. \emph{The biharmonic Neumann problem in Lipschitz domains.} Acta Math. \textbf{194} (2005), no. 2, 217--279.

\bibitem{W56}
H.F. Weinberger. \emph{An isoperimetric inequality for the $N$-dimensional free membrane problem.}
J. Rational Mech. Anal. \textbf{5} (1956), 633--636.

\bibitem{RW74}
R. Weinstock. \textrm{Calculus of Variations}, Dover, New York, 1974. (Reprint of 1952 edition.)

\end{thebibliography}
\end{document}